\documentclass{amsart}
\usepackage{graphicx}
\usepackage{amssymb}

\usepackage{longtable}

\newtheorem{theorem}{Theorem}[section]
\newtheorem{proposition}[theorem]{Proposition}
\newtheorem{lemma}[theorem]{Lemma}

\newtheorem{corollary}[theorem]{Corollary}
\newtheorem{conjecture}[theorem]{Conjecture}

\theoremstyle{definition}
\newtheorem{definition}[theorem]{Definition}
\newtheorem{example}[theorem]{Example}

\newcommand{\Z}{\Bbb Z}

\newcommand{\sign}{{\rm sign}}

\theoremstyle{remark}

\numberwithin{equation}{section}
\begin{document}
\title[Twisted Alexander polynomials of knots]{Twisted Alexander polynomials of knots associated to the regular representations of finite groups}

\author{Takayuki Morifuji and Masaaki Suzuki}



\begin{abstract}
The twisted Alexander polynomial of a knot is defined associated to a linear representation of the knot group. 
If there exists a surjective homomorphism of a knot group onto a finite group, then we obtain a representation of the knot group 
by the composition of the surjective homomorphism and the regular representation of the finite group. 
In this paper, we provide several formulas of the twisted Alexander polynomial of a knot associated to such representations 
in terms of the Alexander polynomial.
\end{abstract}

\address{Department of Mathematics, 
Hiyoshi Campus, Keio University, 
Yokohama 223-8521, Japan}

\email{morifuji@z8.keio.jp}

\address{Department of Frontier Media Science, 
Meiji University, 
4-21-1 Nakano, Nakano-ku, Tokyo 
164-8525, Japan}

\email{mackysuzuki@meiji.ac.jp}

\maketitle

\section{Introduction}

The Alexander polynomial is one of the most fundamental invariants of a knot. 
The twisted Alexander polynomial of a knot is defined with a linear representation of the knot group, 
which can be considered as a generalization of the Alexander polynomial, 
where the knot group is the fundamental group of the complement of a knot in the $3$-sphere. 
Note that the twisted Alexander polynomial associated to the trivial representation is essentially equivalent to the Alexander polynomial. 
It is known that we can extract more various information from the twisted Alexander polynomial than the Alexander polynomial. 
For example, the twisted Alexander polynomial detects the trivial knot, the trefoil, and the figure eight knot (see \cite{SW06-2}, \cite{FV07}, \cite{FV17}). 
Besides, many papers investigate the twisted Alexander polynomials associated to several types of representations. 

A finite group admits the regular representation induced by the left action on itself. 
The regular representation of a finite group contains all the irreducible representations of the finite group. 
To be precise, the irreducible decomposition of the regular representation of a finite group is 
the direct sum of all the irreducible representations with multiplicity its dimension. 
If a knot group admits a surjective homomorphism onto a finite group, 
we obtain a representation of the knot group by the composition of the surjective homomorphism onto the finite group and the regular representation. 
In this paper, we will discuss the twisted Alexander polynomial of a knot associated to such a representation. 
More precisely, the twisted Alexander polynomials for some finite groups modulo a prime number can be expressed in terms of the Alexander polynomial. 
As a corollary, it follows that the twisted Alexander polynomials of knots for all the finite groups of order less than $24$ are not zero, 
which is already shown in \cite{ishikawamorifujisuzuki}.

This paper is organized as follows. 
In Section 2, 
we quickly recall the definition of the twisted Alexander polynomial and its basic properties. 
In Section 3, 4, 5, 
we will give explicit formulas of the twisted Alexander polynomials for the dihedral group, a kind of metacyclic group $G(m,p|k)$, and the dicyclic group. 
In Section 6, as a concluding remark, 
we will consider the other finite groups of order less than $24$.

Throughout this paper, 
we adopt Rolfsen's table \cite{Rolfsen76-1} to represent a prime knot with $10$ or fewer crossings. 

\section{Twisted Alexander polynomials}

In \cite{Wada94-1}, the twisted Alexander polynomial is defined for finitely presentable group. 
In this section, we review quickly the definition of the twisted Alexander polynomial of a knot. 
See \cite{Wada94-1} for the precise definition and more general settings. 

Let $G(K)$ be the knot group of a knot $K$ in the $3$-sphere $S^3$, namely, the fundamental group of the exterior of $K$ in $S^3$. 
We take the Wirtinger presentation 
\[
G(K) = \langle 
x_1, x_2, \ldots, x_m \, | \, r_1, r_2, \ldots, r_{m-1}
\rangle 
\]
according to a diagram of $K$, where $x_i$ represents a meridian of $K$. 
The abelianization $\phi : G(K) \to {\mathbb Z} \simeq \langle t \rangle$ sends $x_i$ to $t$. 

\begin{definition}
The twisted Alexander polynomial of $K$ 
associated to a linear representation $\rho : G(K) \to GL(n, {\mathbb C})$ is defined by  
\[
\Delta_K^\rho (t) = \frac{\det \Big((\rho \otimes \phi) \frac{\partial r_i}{\partial x_j} \Big)_{1 \leq i,j \leq m-1}}{\det \big((\rho \otimes \phi) (x_m - 1) \big)}
\]
where $\frac{\partial}{\partial x_j}$ is the Fox derivation with respect to $x_j$ 
and $\rho, \phi$ are extended linearly to ${\mathbb Z}[G(K)]$. 
\end{definition}

Remark that the twisted Alexander polynomial is defined up to multiplication of $a t^k$ for some $a \in {\mathbb C}^\times, k \in {\mathbb Z}$ 
and that the twisted Alexander polynomial is not always a genuine polynomial. 
For example, the twisted Alexander polynomial associated to the $n$-dimensional trivial representation ${\bf 1} : G(K) \to GL(n,{\mathbb C})$
can be expressed as  
\[
\Delta_K^{\bf 1} (t) = \left( \frac{\Delta_K (t)}{t - 1} \right)^n 
\]
where $\Delta_K(t)$ is the (classical) Alexander polynomial of $K$. 
We have the following fundamental properties on the twisted Alexander polynomials. 

\begin{lemma}\label{lem-tapproperties}
\begin{enumerate}
\item
If $\rho_1, \rho_2$ are conjugate representations of $G(K)$, then 
\[
\Delta_K^{\rho_1} (t) = \Delta_K^{\rho_2} (t) .
\]
In this sense, we do not distinguish conjugate representations in this paper. 
\item
If $\rho, \rho_1, \rho_2$ are representations of $G(K)$ of the form 
\[
\rho(g) = 
\begin{pmatrix}
\rho_1(g) & * \\
0 & \rho_2 (g)
\end{pmatrix}
\]
for any $g \in G(K)$, then 
\[
\Delta_K^{\rho} (t) = \Delta_K^{\rho_1} (t) \cdot \Delta_K^{\rho_2} (t) .
\]
\item 
Suppose that $\rho : G(K) \to GL(1,{\mathbb C})\, (\simeq {\mathbb C}^\times)$ is a one-dimensional representation defined by $\rho(x_i) = \alpha$ 
and that $\tau : G(K) \to GL(n,{\mathbb C})$ is an $n$-dimensional representation. 
Then the twisted Alexander polynomial of $K$ associated to the representation $\rho \otimes \tau : G(K) \to GL(n,{\mathbb C})$ 
obtained by $(\rho \otimes \tau) (x_i) = \alpha \, \tau(x_i)$ is 
\[
\Delta_K^{\rho \otimes \tau} (t) = \Delta_K^{\tau} (\alpha t).
\]
In particular, if $\tau$ is the one-dimensional trivial representation, 
namely, $\tau(x_i) = 1$, 
then we have 
\[
\Delta_K^{\rho \otimes \tau} (t) = \Delta_K^{\rho} (t) = \frac{\Delta_K (\alpha t)}{\alpha t - 1} .
\]
\end{enumerate}
\end{lemma}

In this paper, 
we take the composition of a surjective homomorphism of $G(K)$ onto a finite group $G$ 
and the regular representation $\rho : G \to GL(|G|,{\mathbb Z})$ of $G$ as a representation of $G(K)$ to obtain the twisted Alexander polynomial of $K$. 
The action of a finite group $G$ on itself induces the regular representation of $G$. 
To be precise, let $V$ be the vector space with basis $\{e_h \,|\, h \in G\}$. 
The action $g\cdot\sum a_h e_h=\sum a_h e_{gh}$ of $G$ on $V$ determines the regular representation of $G$. 
The regular representation of a finite group $G$ is the most important representation of $G$ in the following sense. 

\begin{lemma}\label{lem-regularrep}
Let $\rho_1, \rho_2, \ldots, \rho_k$ be all the irreducible representations of a finite group $G$.  
Then the regular representation $\rho : G\to GL(|G|,\Z)$ of $G$ admits the irreducible decomposition as 
\[
\rho=\rho_1^{\oplus \dim \rho_1} \oplus \rho_2^{\oplus \dim \rho_2} \oplus \cdots \oplus \rho_k^{\oplus \dim \rho_k}.
\] 
Moreover, for a surjective homomorphism $f : G(K) \to G$ and the regular representation $\rho$, we have   
\[
\Delta_K^{\rho \circ f} (t)=\prod_{j=1}^k\left(\Delta_K^{\rho_j \circ f}(t)\right)^{\dim \rho_j}.
\]
\end{lemma}

It is known that the number of the irreducible representations of a finite group $G$ is equal to the number of the conjugacy classes of $G$. 
Note that if a finite group $G$ is not normally generated by one element, then any knot group never admit a surjective homomorphism onto $G$, 
since a knot group is normally generated by its meridian. 

First, we consider abelian groups. 
It is known that every irreducible representation of a finite abelian group is one-dimensional. 
Then the twisted Alexander polynomial $\Delta_K^{\rho \circ f} (t)$ can be described 
in terms of the Alexander polynomial $\Delta_K(t)$. 

\begin{proposition}\label{prop-abelian}
\begin{enumerate}
\item Let $A$ be an abelian group. 
There exists a surjective homomorphism $f :G(K) \to A$ if and only if $A$ is a cyclic group. 
\item Let $C_n$ be a cyclic group of order $n$. 
For a surjective homomorphism $f :G(K) \to C_n$ and 
the regular representation $\rho : C_n \to GL(n,{\mathbb Z})$, 
the twisted Alexander polynomial is given by  
\[
\Delta_K^{\rho \circ f} (t) = 
\prod_{j=1}^n \left(\frac{\Delta_K (\omega^j t)}{\omega^j t - 1} \right)
\]
where $\omega=e^{2\pi i/n} \in {\mathbb C}$.  
\end{enumerate}
\end{proposition}

\begin{proof}
\begin{enumerate}
\item 
The abelianization of a knot group is the infinitely cyclic group ${\mathbb Z}$. 
Then for any cyclic group $C_n$, there exists a surjective homomorphism $f  : G(K) \to {\mathbb Z} \to C_n$ (for any knot $K$). 
Conversely, non-cyclic abelian groups are not normally generated by one element. 
Then any knot group does not admit a surjective homomorphism onto non-cyclic abelian groups. 
\item 
Since $f$ is surjective, the meridian $x_i$ is mapped to a generator $g$ of $C_n$, namely, $f(x_i) = g$.
The cyclic group $C_n$ admits $n$ one-dimensional irreducible representations $\rho_1, \rho_2, \ldots, \rho_n$. 
We may assume that $\rho_j (g) = \omega^j$ without loss of generality,  
by changing the numbering of $j$ of $\rho_j$ if necessary. 
Then by Lemma \ref{lem-tapproperties} (iii) and Lemma \ref{lem-regularrep}, we get the desired statement. 
\end{enumerate}
\end{proof}

In this paper, 
we use several well-known identities on binomial coefficients. 

\begin{proposition}\label{prop-dihedralfact}
\begin{enumerate}
\item {}[Lucas's theorem] Let $p$ be a prime number and $m,n$ non-negative integers. 
If the base $p$ expansions of $m,n$ are given by
\begin{align*}
m &= m_l \, p^l + m_{l-1} \, p^{l-1} + \cdots + m_1 \, p + m_0 \\
n &= n_l \, p^l + n_{l-1} \, p^{l-1} + \cdots + n_1 \, p + n_0 , 
\end{align*}
then 
\[
\left( 
\begin{array}{c} 
m \\ 
n
\end{array}
\right) 
\equiv
\prod_{k=0}^l
\left( 
\begin{array}{c} 
m_k \\ 
n_k
\end{array}
\right) 
\mod p .
\]
\item {}[Pascal's rule] For positive integers $m,n$, we have 
\[
\left( 
\begin{array}{c} 
m \\ 
n
\end{array}
\right) 
=
\left( 
\begin{array}{c} 
m-1 \\ 
n
\end{array}
\right)  + \left( 
\begin{array}{c} 
m-1 \\ 
n-1
\end{array}
\right) .
\]
\item {}[Vandermonde's identity] For non-negative integers $m,n,r$, we have 
\[
\left( 
\begin{array}{c} 
m + n\\ 
r
\end{array}
\right) 
= \sum_{k=0}^r
\left( 
\begin{array}{c} 
m\\ 
k
\end{array}
\right) 
\left( 
\begin{array}{c} 
n\\ 
r - k
\end{array}
\right),  
\]
in particular, 
\[
\left( 
\begin{array}{c} 
m + n\\ 
n
\end{array}
\right) 
= \sum_{k=0}^n
\left( 
\begin{array}{c} 
m\\ 
k
\end{array}
\right) 
\left( 
\begin{array}{c} 
n\\ 
k
\end{array}
\right) .  
\]
\item {}[Euler's finite difference theorem] \cite[(6.16), (6.24)]{gould} For a non-negative integers $n$, we have
\[
\sum_{k=0}^n
(-1)^k
\left( 
\begin{array}{c} 
n\\ 
k
\end{array}
\right) 
f(k)
=
\left\{
\begin{array}{ll}
0 & r < n \\ 
(-1)^n \, n! \, a_n & r = n
\end{array}
\right. ,
\]
where $f(x) = \sum_{j=0}^r a_j x^j$. In particular, 
if $f(k) = \left( \begin{array}{c} a k-2\\ r \end{array} \right)$ for non-zero $a$, then 
\[
\sum_{k=0}^n
(-1)^k
\left( 
\begin{array}{c} 
n\\ 
k
\end{array}
\right) 
\left( 
\begin{array}{c} 
a k-2\\ 
r
\end{array}
\right)
=
\left\{
\begin{array}{ll}
0 & r < n \\ 
(-1)^n a^n & r = n
\end{array}
\right. .
\]
\end{enumerate}
\end{proposition}

\section{Dihedral groups $D_n$}

In this section, we consider the dihedral group $D_n$ of degree $n$ with presentation  
\[
D_n = \langle a,b \, | \, a^n = b^2 =1, b a b = a^{-1} \rangle. 
\]
If $n$ is even, $D_n$ is not normally generated by one element. 
Then we deal with the dihedral group of odd degree in this paper. 
In this case, $D_n$ is normally generated by $b$. 

\begin{lemma}\label{lem-dihedaralcharacter}
Let $n$ be an odd natural number. 
There are $2$ one-dimensional irreducible representations of $D_n$, 
namely, the trivial representation ${\bf 1} : D_n \to \{1\}$  and the alternating representation $\sign : D_n \to \{1,-1\}$ defined by $\sign (a) = 1, \sign (b)= -1$. 
The other irreducible representations of $D_n$ are two-dimensional representations $\rho_1,\rho_2,\ldots, \rho_{\frac{n-1}{2}}$ defined by 
\[
\rho_j(a) = 
\begin{pmatrix}
\omega^j & 0 \\
0 & \omega^{- j} 
\end{pmatrix}, \quad 
\rho_j(b) = 
\begin{pmatrix}
0 & 1 \\
1 & 0 
\end{pmatrix} 
\]
where $\omega=e^{2\pi i/n}$.
These characters are given by Table \ref{characterDn},   
where $1\leq k \leq \frac{n-1}{2},1\leq j \leq \frac{n-1}{2}$. 

\begin{table}[h]
\begin{tabular}{c|ccc}
& $1$ & $b$ & $a^k$\\
\hline
$\chi_{{\bf 1}}$ & $1$ & $1$ & $1$\\
$\chi_\sign$ & $1$ & $-1$ & $1$\\
$\chi_{\rho_j}$ & $2$ & $0$ & $\omega^{j k} + \omega^{-j k} $ 
\end{tabular}
\caption{Character table of $D_n$}
\label{characterDn}
\end{table}
\end{lemma}

The twisted Alexander polynomials for the dihedral group can be expressed by using the Alexander polynomial.  

\begin{theorem}\label{thm-dihedral}
Let $p$ be an odd prime number and $q=p^n$. 
We denote 
by $\rho : D_{q} \to GL(2 q,{\mathbb Z})$ the regular representation of $D_q$. 
If there exists a surjective homomorphism $f :G(K) \to D_q$, then
\[
\Delta_K^{\rho \circ f} (t) \equiv 
\left(
\frac{\Delta_K (t)}{t-1} \cdot \frac{\Delta_K (-t)}{t+1}
\right)^{q} \mod p .
\]
\end{theorem}

Boden-Friedl in \cite{BF14-1} discussed the twisted Alexander polynomial for $D_p$, 
namely when $n=1$ in the above theorem,  
associated to the restriction to $D_p$ of the permutation representation of the symmetric group $S_p$. 
For the proof of Theorem \ref{thm-dihedral}, we prepare some identities of binomial coefficients. 

\begin{lemma}\label{lem-dihedralclaim}
Let $p$ be a prime number and $q= p^n$ , then we have the following:  
\begin{enumerate}
\item 
\[
\left( 
\begin{array}{c} 
q -1\\ 
k
\end{array}
\right) 
\equiv 
(-1)^k \mod p ,
\]
\item
\[ 
\sum_{k=1}^j (-1)^{j-k} 
\left( 
\begin{array}{c} 
m \\ 
k -1
\end{array}
\right) 
=
\left( 
\begin{array}{c} 
m - 1 \\ 
j - 1
\end{array}
\right), \quad m \geq 1 , 
\]
\item 
\[
(-1)^{j-1} 
\left( 
\begin{array}{c} 
i+j-2\\ 
j-1
\end{array}
\right) 
\equiv
\left( 
\begin{array}{c} 
q - i \\ 
j-1
\end{array}
\right)
\mod p, \quad 1 \leq i,j \leq q .
\]
\end{enumerate}
\end{lemma}

\begin{proof}
\begin{enumerate}
\item First, the coefficient of $x^k$ in $(x+1)^{q-1}$ is $\left( \begin{array}{c} q - 1 \\ k \end{array} \right)$.
It is easy to see that 
\[
(x+1)^q \equiv x^q + 1 \mod p .
\]
Then we have 
\[
(x+1)^{q-1} \equiv \frac{x^q + 1}{x+1} \equiv 
x^{q-1} - x^{q-2} + x^{q-3} - \cdots - x + 1 \mod p .
\]
Therefore the coefficient of $x^k$ in $(x+1)^{q-1}$ is $(-1)^k$. 
\item By Pascal's rule, we obtain 
\begin{align*}
\sum_{k=1}^j (-1)^{j-k} 
\left( 
\begin{array}{c} 
m \\ 
k -1
\end{array}
\right) 
=& 
(-1)^{j-1}
\left( \begin{array}{c} m \\ 0 \end{array} \right) \\
&+ (-1)^{j-2}
\left(
\left( \begin{array}{c} m - 1 \\ 1 \end{array} \right) +
\left( \begin{array}{c} m - 1 \\ 0 \end{array} \right)
\right) \\ 
&+ (-1)^{j-3}
\left(
\left( \begin{array}{c} m - 1 \\ 2 \end{array} \right) +
\left( \begin{array}{c} m - 1 \\ 1 \end{array} \right)
\right) \\
&+ \cdots 
+ (-1)^0
\left(
\left( \begin{array}{c} m - 1 \\ j-1 \end{array} \right) +
\left( \begin{array}{c} m - 1 \\ j-2 \end{array} \right)
\right) \\
&= 
\left( 
\begin{array}{c} 
m - 1 \\ 
j - 1
\end{array}
\right).
\end{align*}
\item We will show this statement by double induction on $i,j$. 
When $j=1$, both sides of the desired statement are $1$. 
When $i=q$ and $j \geq 2$, the left hand side is 
\begin{align*}
(-1)^{j-1} 
\left( \begin{array}{c} q + j - 2 \\ j-1 \end{array} \right)
&= 
(-1)^{j-1}  \frac{(q+ j -2) (q + j - 3) \cdots (q+1) q}{(j-1)!} \\
& \equiv 0 \mod p  
\end{align*}
and the right hand side is also $\left( \begin{array}{c} 0 \\ j-1 \end{array} \right) = 0$. 
Then if $i=q$ or $j=1$, then the statement holds. 
Suppose that the statement is true for $i \geq k_1$ or $j \leq k_2$.
When $i= k_1 - 1$ and $j = k_2 + 1$, the left hand side is 
\begin{align*}
& (-1)^{k_2} \left( \begin{array}{c} k_1+k_2 - 2\\ k_2 \end{array} \right) \\
&= 
(-1)^{k_2} 
\left(
\left( \begin{array}{c} k_1+ k_2 - 1 \\ k_2 \end{array} \right) 
-\left( \begin{array}{c} k_1+k_2-2 \\ k_2-1 \end{array} \right) 
\right) \quad \mbox{ by Pascal's rule} \\
&= 
(-1)^{k_2} 
\left( \begin{array}{c} k_1+k_2 -1 \\ k_2 \end{array} \right) 
+ (-1)^{k_2-1} 
\left( \begin{array}{c} k_1+k_2 -2\\ k_2 -1\end{array} \right)  \\
& \equiv 
\left( \begin{array}{c} q - k_1  \\ k_2 \end{array} \right)  
+ \left( \begin{array}{c} q - k_1 \\ k_2 -1\end{array} \right)  
\quad \mbox{ by assumption of induction} \\
& \equiv 
\left( \begin{array}{c} q - (k_1-1) \\ k_2\end{array} \right)  \mod p
\quad \mbox{ by Pascal's rule}.
\end{align*}
The last expression is the right hand side of the desired statement for $i= k_1 - 1$ and $j = k_2 + 1$. 
This completes the proof. 
\end{enumerate}
\end{proof}

\begin{proof}[Proof of Theorem \ref{thm-dihedral}]
The dihedral group $D_q$ can be naturally embedded in the symmetric group $S_q$. 
We denote by $\bar{\rho}_1 : D_q \to GL(q, {\mathbb Z})$ the restriction of the permutation representation of $S_q$, 
namely,  
\[
\bar{\rho}_1(a) = 
\begin{pmatrix}
0 &    & & & 1 \\
1 & 0 & &\text{\huge{0}} & \\
  &  1 & \ddots & & \\
  &     & \ddots & \ddots & \\
& \text{\huge{0}}     &   &  1& 0  
\end{pmatrix}, \quad 
\bar{\rho}_1(b) = 
\begin{pmatrix}
& \text{\huge{0}}  & & 1 \\
 & & 1 & \\
 &  \rotatebox{-45}{\vdots} &  & \\
1   &  &  \text{\huge{0}} & 
\end{pmatrix}. 
\]
By Lemma \ref{lem-dihedaralcharacter},
the character of $\bar{\rho}_1$  
\begin{align*}
\chi_{\bar{\rho}_1} &= (q,1,0,0,\ldots,0) \\
& = (1,1,1,1,\ldots,1) + \sum_{j=1}^{\frac{q-1}{2}} (2,0,\omega^j + \omega^{-j}, \omega^{2j} + \omega^{-2j}, \ldots, \omega^{\frac{q-1}{2} j} + \omega^{-\frac{q-1}{2} j})  
\end{align*}
provides the irreducible decomposition  
\[
\bar{\rho}_1 = {\bf 1} \oplus \rho_1 \oplus \rho_2 \oplus \cdots \oplus  \rho_{\frac{q-1}{2}} .
\]
Similarly, we define $\bar{\rho}_2 : D_q \to GL(q, {\mathbb Z})$ by  
\[
\bar{\rho}_2(a) = 
\begin{pmatrix}
0 &    & & & 1 \\
1 & 0 & &\text{\huge{0}} & \\
  &  1 & \ddots & & \\
  &     & \ddots & \ddots & \\
& \text{\huge{0}}     &   &  1& 0  
\end{pmatrix}, \quad 
\bar{\rho}_2(b) = 
\begin{pmatrix}
& \text{\huge{0}}  & & -1 \\
 & & -1 & \\
 &  \rotatebox{-45}{\vdots} &  & \\
-1   &  &  \text{\huge{0}} & 
\end{pmatrix}. 
\]
It is easy to see that $\bar{\rho}_2$ is a representation with $\chi_{\bar{\rho}_2} = (q,-1,0,0,\ldots,0)$. 
Then we obtain the irreducible decompositon:  
\[
\bar{\rho}_2 =  \sign \oplus \rho_1 \oplus \rho_2 \oplus \cdots \oplus  \rho_{\frac{q-1}{2}} .
\]
Therefore the irreducible decomposition of the regular representation 
\[
\rho = {\bf 1} \oplus \sign \oplus {\rho_1}^{\oplus2} \oplus {\rho_2}^{\oplus2} \oplus \cdots \oplus  {\rho_{\frac{q-1}{2}}}^{\oplus2} 
\]
shows $\rho = \bar{\rho}_1 \oplus \bar{\rho}_2$.
By Lemma \ref{lem-regularrep}, we obtain   
\[ 
\Delta_K^{\rho \circ f} (t) = \Delta_K^{\bar{\rho}_1 \circ f} (t) \cdot \Delta_K^{\bar{\rho}_2 \circ f} (t).
\]

First, we compute  $\Delta_K^{\bar{\rho}_1 \circ f} (t)$. 
Let $A_{p,n}, \tau(a), \tau (b)$ be $q \times q$ matrices whose $(i,j)$-entries ${A_{p,n}}_{(i,j)}, \tau(a)_{(i,j)}, \tau (b)_{(i,j)}$ are given by 
\begin{align*}
{A_{p,n}}_{(i,j)} &= 
\left( 
\begin{array}{c} 
i-1 \\ 
j-1 
\end{array}
\right) \mod p, \\
\tau(a)_{(i,j)} &= 
\left\{
\begin{array}{cl}
1 & i \leq j, \, j-i \mbox{ is even} \\
-1 & i \leq j, \, j-i \mbox{ is odd} \\
0 & i > j
\end{array}
\right. , \\ 
\tau (b)_{(i,j)} &=
(-1)^{j-1}
\left( 
\begin{array}{c} 
j-1 \\ 
i-1 
\end{array}
\right) \mod p 
\end{align*}
respectively. 
Remark that $A_{p,n}$ is a non-singular matrix with determinant $1$. 
By Lucas's theorem, ${A_{p,n}}_{(i,j)}, \tau (b)_{(i,j)}$ can be expressed as 
\begin{align*}
{A_{p,n}}_{(i,j)} &= 
\prod_{k=0}^{n-1}
\left( 
\begin{array}{c} 
i_k \\ 
j_k 
\end{array}
\right) \mod p, \\
\tau (b)_{(i,j)} &=
(-1)^{j-1}
\prod_{k=0}^{n-1}
\left( 
\begin{array}{c} 
j_k \\ 
i_k 
\end{array}
\right) \mod p
\end{align*}
where $i_k, j_k$ are given by the base $p$ expansions of $i-1,j-1$:
\begin{align*}
i-1 &= i_{n-1} \, p^{n-1} + i_{n-2} \, p^{n-2} + \cdots + i_1 \, p + i_0, \\
j-1 &= j_{n-1} \, p^{n-1} + j_{n-2} \, p^{n-2} + \cdots + j_1 \, p + j_0 . 
\end{align*}
In particular, both $\tau(a), \tau (b)$ are upper triangular matrices. 
All the diagonal entries of $\tau (a)$ are $1$ and those of $\tau (b)$ are $1$ or $-1$ alternatively. 
We have 
\begin{align*}
{(\bar{\rho}_1 (a) \cdot A_{p,n})}_{(1,j)} &\equiv {A_{p,n}}_{(q,j)} \equiv  
\left( 
\begin{array}{c} 
q-1 \\ 
j-1 
\end{array}
\right)
\equiv 
(-1)^{j-1} \quad \mbox{ by Lemma \ref{lem-dihedralclaim} (i)} \\
& \equiv \tau (a) _{(1,j)} \equiv  {(A_{p,n} \cdot \tau(a))}_{(1,j)} \mod p, \\
{(\bar{\rho}_1 (a) \cdot A_{p,n})}_{(i,j)} &\equiv {A_{p,n}}_{(i-1,j)} \equiv  
\left( 
\begin{array}{c} 
i-2 \\ 
j-1 
\end{array}
\right) \\
& \equiv 
\sum_{k=1}^j (-1)^{j-k}
 \left( 
\begin{array}{c} 
i-1 \\ 
k-1 
\end{array}
\right)
\quad \mbox{ by Lemma \ref{lem-dihedralclaim} (ii)} \\
& \equiv  {(A_{p,n} \cdot \tau(a))}_{(i,j)} \mod p \quad (2 \leq i \leq q), \\
{(\bar{\rho}_1 (b) \cdot A_{p,n})}_{(i,j)} &\equiv  
\left( 
\begin{array}{c} 
q-i \\ 
j-1 
\end{array}
\right) \\
& \equiv 
(-1)^{j-1}
\left( 
\begin{array}{c} 
i+j-2 \\ 
j-1 
\end{array}
\right)
\quad \mbox{ by Lemma \ref{lem-dihedralclaim} (iii)} \\
& \equiv 
(-1)^{j-1}
\sum_{k=1}^q 
 \left( 
\begin{array}{c} 
i-1 \\ 
k-1 
\end{array}
\right)
 \left( 
\begin{array}{c} 
j-1 \\ 
k-1 
\end{array}
\right)
\quad \mbox{ by Vandermonde's identity} \\
& \equiv 
\sum_{k=1}^q 
 \left( 
\begin{array}{c} 
i-1 \\ 
k-1 
\end{array}
\right)
(-1)^{j-1}
\left( 
\begin{array}{c} 
j-1 \\ 
k-1 
\end{array}
\right) \\
& \equiv  {(A_{p,n} \cdot \tau(b))}_{(i,j)} \mod p. \\
\end{align*}  
Then the representation $\bar{\rho}_1$ is conjugate to $\tau$, namely, 
\[
{A_{p,n}}^{-1} \bar{\rho}_1 (a) A_{p,n} \equiv \tau (a), \quad 
{A_{p,n}}^{-1} \bar{\rho}_1 (b) A_{p,n} \equiv \tau (b) \mod p
\]
and the twisted Alexander polynomial associated to $\bar{\rho}_1 \circ f$ modulo $p$ is the same 
as that associated to $\tau \circ f$ modulo $p$. 
Since $f : G(K) \to D_q$ is surjective, $f(x_i)$ is $b$ or the conjugate of $b$. 
For any element $g$ in the conjugacy class of $b$, we have 
\[
\tau (g) \equiv 
\begin{pmatrix}
1 & & & & &  \\
& -1 & & & \text{\huge{*}} & \\
& & 1 & & & \\
& & & \ddots & & \\
&  \text{\huge{0}}& & & -1 & \\
& & & & & 1
\end{pmatrix}
\mod p . 
\] 
Therefore by Lemma \ref{lem-tapproperties} (ii) (iii), the twisted Alexander polynomial associated to $\bar{\rho}_1 \circ f$ is 
\[
\Delta_K^{\bar{\rho}_1 \circ f} (t) \equiv 
\Delta_K^{\tau \circ f} (t) \equiv 
\left(
\frac{\Delta_K (t)}{t-1} 
\right)^{\frac{q+1}{2}} 
\cdot
\left(
\frac{\Delta_K (-t)}{t+1}
\right)^{\frac{q-1}{2}} 
\mod p .
\]
By $\bar{\rho}_2 (a) = \bar{\rho}_1 (a)$ and $\bar{\rho}_2 (b) = - \bar{\rho}_1 (b)$, 
$\bar{\rho}_2 (g)$ is conjugate to $- \tau(g)$ for any element $g$ in the conjugacy class of $b$
and then 
\[
\Delta_K^{\bar{\rho}_2 \circ f} (t) \equiv 
\left(
\frac{\Delta_K (t)}{t-1} 
\right)^{\frac{q-1}{2}} 
\cdot 
\left(\frac{\Delta_K (-t)}{t+1}
\right)^{\frac{q+1}{2}} 
\mod p .
\]

Hence we obtain the desired statement 
\[
\Delta_K^{\rho \circ f} (t) = \Delta_K^{\bar{\rho}_1 \circ f} (t) \cdot  \Delta_K^{\bar{\rho}_2 \circ f} (t) \equiv
\left(
\frac{\Delta_K (t)}{t-1} \cdot \frac{\Delta_K (-t)}{t+1}
\right)^{q} \mod p .
\]
\end{proof}

\begin{example}
When $q=3^2$, the matrices mentioned in the above proof are 
\begin{align*}
A_{3,2} &= 
\left(
\begin{array}{ccccccccc}
1 & 0 & 0 & 0 & 0 & 0 & 0 & 0 & 0 \\ 
1 & 1 & 0 & 0 & 0 & 0 & 0 & 0 & 0 \\ 
1 & 2 & 1 & 0 & 0 & 0 & 0 & 0 & 0 \\ 
1 & 0 & 0 & 1 & 0 & 0 & 0 & 0 & 0 \\ 
1 & 1 & 0 & 1 & 1 & 0 & 0 & 0 & 0 \\ 
1 & 2 & 1 & 1 & 2 & 1 & 0 & 0 & 0 \\ 
1 & 0 & 0 & 2 & 0 & 0 & 1 & 0 & 0 \\ 
1 & 1 & 0 & 2 & 2 & 0 & 1 & 1 & 0 \\ 
1 & 2 & 1 & 2 & 1 & 2 & 1 & 2 & 1  
\end{array}
\right), \\
\tau(a) &= 
\left(
\begin{array}{ccccccccc}
1 & 2 & 1 & 2 & 1 & 2 & 1 & 2 & 1 \\ 
0 & 1 & 2 & 1 & 2 & 1 & 2 & 1 & 2 \\ 
0 & 0 & 1 & 2 & 1 & 2 & 1 & 2 & 1 \\ 
0 & 0 & 0 & 1 & 2 & 1 & 2 & 1 & 2 \\ 
0 & 0 & 0 & 0 & 1 & 2 & 1 & 2 & 1 \\ 
0 & 0 & 0 & 0 & 0 & 1 & 2 & 1 & 2 \\ 
0 & 0 & 0 & 0 & 0 & 0 & 1 & 2 & 1 \\ 
0 & 0 & 0 & 0 & 0 & 0 & 0 & 1 & 2 \\ 
0 & 0 & 0 & 0 & 0 & 0 & 0 & 0 & 1  
\end{array}
\right), \\
\tau(b) &= 
\left(
\begin{array}{ccccccccc}
1 & 2 & 1 & 2 & 1 & 2 & 1 & 2 & 1 \\ 
0 & 2 & 2 & 0 & 1 & 1 & 0 & 2 & 2 \\ 
0 & 0 & 1 & 0 & 0 & 2 & 0 & 0 & 1 \\ 
0 & 0 & 0 & 2 & 1 & 2 & 2 & 1 & 2 \\ 
0 & 0 & 0 & 0 & 1 & 1 & 0 & 1 & 1 \\ 
0 & 0 & 0 & 0 & 0 & 2 & 0 & 0 & 2 \\ 
0 & 0 & 0 & 0 & 0 & 0 & 1 & 2 & 1 \\ 
0 & 0 & 0 & 0 & 0 & 0 & 0 & 2 & 2 \\ 
0 & 0 & 0 & 0 & 0 & 0 & 0 & 0 & 1  
\end{array}
\right). 
\end{align*}
For any element $g$ in the conjugacy class of $b$,  
$\tau (g)$ is an upper triangular matrix whose diagonal entries are $(1,2,1,2,1,2,1,2,1)$. 

For example, for a surjective homomorphism $f : G(3_1) \to D_{3^2}$ and the regular representation $\rho : D_{3^2} \to GL(18, {\mathbb Z})$, 
a direct computation shows that 
\[
\Delta_{3_1}^{\rho \circ f} (t) = 
\frac{(t+1)^6 (t^2 - t + 1)^6 (t-1)^6 (t^2 + t + 1)^6}{ (t+1)^3 (t^2 - t + 1)^3 (t-1)^3 (t^2 + t + 1)^3}
\equiv 
\frac{(t+1)^{18} (t-1)^{18}}{ (t+1)^9 (t-1)^9} \mod 3 .
\]
It is well-known that the Alexander polynomial of $3_1$ is $\Delta_{3_1} (t) = t^2 - t + 1$. 
Then 
\[
\left(
\frac{\Delta_{3_1} (t)}{t-1} \cdot \frac{\Delta_{3_1} (-t)}{t+1}
\right)^{3^2} 
= 
\left(
\frac{t^2-t+1}{t-1} \cdot \frac{t^2 + t + 1}{t+1}
\right)^{9} 
\equiv 
\left(
\frac{(t+1)^{2} (t-1)^{2}}{ (t+1) (t-1)}
\right)^9 \mod 3 .
\]
\end{example}

Combined with Theorem \ref{thm-dihedral}, Lemma \ref{lem-tapproperties}, Proposition \ref{prop-abelian}, 
we get the following corollary. 

\begin{corollary}\label{cor-dihedralcyclic}
Let $p$ be an odd prime number and $q=p^n$. 
We denote 
by $\rho : D_{q} \times C_m \to GL(2 m q,{\mathbb Z})$ the regular representation of $D_q \times C_m$. 
If there exists a surjective homomorphism $f :G(K) \to D_q \times C_m$, then
\[
\Delta_K^{\rho \circ f} (t) \equiv 
\left(
\prod_{j=1}^m \left(\frac{\Delta_K (\omega^j t)}{\omega^j t - 1} \right) \cdot 
\prod_{j=1}^m \left(\frac{\Delta_K (- \omega^j t)}{\omega^j t + 1} \right)
\right)^{q} \mod p 
\]
where $\omega=e^{2\pi i/m} \in {\mathbb C}$.
\end{corollary}

Remark that if $m$ is even, 
then $D_{q} \times C_m$ is not normally generated by one element   
and any knot group never admits a surjective homomorphism $f :G(K) \to D_q \times C_m$.

\section{Metacyclic groups $G(m,p|k)$}

Let $G(m,p | k)$ be a finite group of order $mp$ defined by the following presentation: 
\[
G(m,p | k) = \langle a,b \, | \, a^p = b^m = 1,\, bab^{-1} = a^k \rangle
\]
where $m \in {\mathbb N}$, $p$ is an odd prime number such that $p \equiv 1$ modulo $m$, 
and $k$ is a primitive $m$-th root of unity modulo $p$.  
In \cite{fox}, Fox called this group $K$-metacyclic group when $m = p-1$. 
Hirasawa-Murasugi and Boden-Friedl discussed the twisted Alexander polynomial for $G(m,p|k)$ 
in \cite{hirasawamurasugi}, \cite{BF14-1} respectively. 
Note that $G(m,p|k)$ is a semi-direct product $C_m \ltimes C_p$ 
and that $G(2,p|p-1)$ is the dihedral group $D_p$ 
(however, $D_{p^n}$ cannot be written as $G(m,p|k)$ for $n > 1$). 

\begin{lemma}\label{lem-charactergmpk}
There are one-dimensional irreducible representations ${\bf 1}, \tau_1,\tau_2,\ldots, \tau_{m-1}$ of $G(m,p|k)$ 
defined by 
\begin{align*}
&{\bf 1} (a) = 1, \quad  {\bf 1} (b) = 1, \\
&\tau_j (a) = 1, \quad \tau_j (b) = \omega^j \quad (1 \leq j \leq m-1) 
\end{align*}
where $\omega=e^{2\pi i/m}$. 
Besides, there are $m$-dimensional irreducible representations $\rho_1,\rho_2,\ldots, \rho_{\bar{p}}$ of $G(m,p|k)$, 
where $p = m \bar{p} + 1$. 
Let $n_1, n_2, \ldots,n_{\bar{p}}$ be natural numbers such that their orders in ${\mathbb F}_p^\times$ divide $\bar{p}$. 
In other words, $\{n_1,n_2,\ldots,n_p \} = \{ n^m \, \, {\rm mod} \, p \, |\,  1 \leq n \leq p-1\}$. 
We denote by $\bar{k}_j$ a primitive $m$-th root of $n_j$ modulo $p$ $(1 \leq j \leq \bar{p})$. 
Then these representations $\rho_1,\rho_2,\ldots, \rho_{\bar{p}}$ are defined by 
\[
\rho_j(a) = 
\begin{pmatrix}
\bar{\omega}^{\bar{k}_j k^0} & & & \text{\huge{$0$}} \\
& \bar{\omega}^{\bar{k}_j k^1} & & \\
& & \ddots & \\
\text{\huge{$0$}} & & & \bar{\omega}^{\bar{k}_j k^{m-1}}
\end{pmatrix}, \quad 
\rho_j(b) = 
\begin{pmatrix}
0 & 1 & & \text{\huge{$0$}} &  \\
 & 0 & 1 && \\
  &   & \ddots & \ddots & \\
  &   \text{\huge{$0$}}  &  & \ddots & 1 \\  
1 &    &   &  & 0  
\end{pmatrix}
\]
where $\bar{\omega}=e^{2\pi i/p}$. 
These characters are given by Table \ref{characterGmpk}. 

\begin{table}[h]
\begin{tabular}{c|ccccccc}
& $1$ & $b$ & $\cdots$ & $b^{m-1}$ & $a^{\bar{k}_1}$ & $\cdots$ & $a^{\bar{k}_{\bar{p}}}$ \\
\hline
$\chi_{{\bf 1}}$ & $1$ & $1$ & $\cdots$ & $1$ & $1$ & $\cdots$ & $1$ \\
$\chi_{\tau_1} $ & $1$ & $\omega$ & $\cdots$ & $\omega^{m-1}$ & $1$ & $\cdots$ & $1$ \\
$\chi_{\tau_2} $ & $1$ & $\omega^{2}$ & $\cdots$ & $\omega^{2(m-1)}$ & $1$ & $\cdots$ & $1$ \\
$\vdots$ & $\vdots$ & $\vdots$ & $\ddots$ & $\vdots$ & $\vdots$ & $\ddots$ & $\vdots$ \\
$\chi_{\tau_{m-1}} $ & $1$ & $\omega^{m-1}$ & $\cdots$ & $\omega^{(m-1)^2}$ & $1$ & $\cdots$ & $1$ \\
$\chi_{\rho_1}$ & $m$ & $0$ & $\cdots$ & $0$ & $\sum_{l=0}^{m-1} \bar{\omega}^{\bar{k}_1 \bar{k}_1 k^l} $ & $\cdots$ & $\sum_{l=0}^{m-1} \bar{\omega}^{\bar{k}_1 \bar{k}_{\bar{p}} k^l} $ \\
$\chi_{\rho_2}$ & $m$ & $0$ & $\cdots$ & $0$ & $\sum_{l=0}^{m-1} \bar{\omega}^{\bar{k}_2 \bar{k}_1 k^l} $ & $\cdots$ & $\sum_{l=0}^{m-1} \bar{\omega}^{\bar{k}_2 \bar{k}_{\bar{p}} k^l} $ \\
$\vdots$ & $\vdots$ & $\vdots$ & $\ddots$ & $\vdots$ & $\vdots$ & $\ddots$ & $\vdots$ \\
$\chi_{\rho_{\bar{p}}}$ & $m$ & $0$ & $\cdots$ & $0$ & $\sum_{l=0}^{m-1} \bar{\omega}^{\bar{k}_{\bar{p}} \bar{k}_1 k^l} $ & $\cdots$ & $\sum_{l=0}^{m-1} \bar{\omega}^{\bar{k}_{\bar{p}} \bar{k}_{\bar{p}} k^l} $
\end{tabular}
\caption{Character table of $G(m,p|k)$}
\label{characterGmpk}
\end{table}
\end{lemma}

\begin{theorem}\label{thm-Gmpk}
Let $k_j \in \{1,2,\ldots,p-1 \}$ be $m$-th roots of unity modulo $p$, namely, ${k_j}^m \equiv 1$ modulo $p$ $(j=1,2,\ldots,m)$ 
and $\omega=e^{2\pi i/m} \in {\mathbb C}$.
We denote by $\rho : G(m,p|k) \to GL(mp,{\mathbb Z})$ the regular representation of $G(m,p|k)$. 
If there exists a surjective homomorphism $f :G(K) \to G(m,p|k)$, then
\[
\Delta_K^{\rho \circ f} (t) \equiv 
\prod_{j=1}^m \left(\frac{\Delta_K (k_j t)}{ k_j t  - 1} \right)^{p-1}  \cdot 
\prod_{j=1}^m \left(\frac{\Delta_K (\omega^j t)}{\omega^j t - 1} \right)
\mod p .
\]
\end{theorem}

\begin{lemma}\label{lem-gmpkclaim}
\begin{enumerate}
\item For a prime number $p$ and $1 \leq i,j \leq p$, we have
\[
\sum_{k=j}^i 
\left( 
\begin{array}{c} 
i-1 \\ 
k-1 
\end{array}
\right)
\left( 
\begin{array}{c} 
p-j \\ 
p-k 
\end{array}
\right) 
=
\left( 
\begin{array}{c} 
(p-j)+(i-1) \\ 
p-1 
\end{array}
\right) .
\]
\item
Let $A_p$ be a $p \times p$ matrix defined as $A_{p,1}$ in Proof of Theorem \ref{thm-dihedral}. 
Then the $(i,j)$-entry of the inverse matrix ${A_p}^{-1}$ is given by 
\[
{{A_{p}}^{-1}}_{(i,j)} \equiv 
\left( 
\begin{array}{c} 
p-j \\ 
p-i 
\end{array}
\right)
\mod p .
\]
\item For $1 \leq i,s \leq p$, we have
\[
\left( 
\begin{array}{c} 
p-s \\ 
p-i 
\end{array}
\right)
\equiv  
(-1)^{i+s}
\left( 
\begin{array}{c} 
i-1 \\ 
s-1 
\end{array}
\right)  \mod p .
\]
\end{enumerate}
\end{lemma}

\begin{proof}
\begin{enumerate}
\item 
The coefficient of $x^{k-1}$ in $(x+1)^{i-1}$ is 
$\left( \begin{array}{c} i-1 \\ k-1 \end{array} \right) $ 
and that of $x^{p-k}$ in $(x+1)^{p-j}$ is 
$\left( \begin{array}{c} p-j \\ p-k \end{array} \right) $.  
Then the left hand side is the coefficient of $x^{p-1}$ in $(x+1)^{(p-j)+(i-1)}$. 
\item Recall the $(i,j)$-entry of $A_{p} (= A_{p,1})$ is $\left( \begin{array}{c} i-1 \\ j-1 \end{array} \right) $. 
Then 
\begin{align*}
\sum_{k=1}^p 
\left( 
\begin{array}{c} 
i-1 \\ 
k-1 
\end{array}
\right)
\left( 
\begin{array}{c} 
p-j \\ 
p-k 
\end{array}
\right) 
&=
\sum_{k=j}^i 
\left( 
\begin{array}{c} 
i-1 \\ 
k-1 
\end{array}
\right)
\left( 
\begin{array}{c} 
p-j \\ 
p-k 
\end{array}
\right) \\
&=
\left( 
\begin{array}{c} 
(p-j)+(i-1) \\ 
p-1 
\end{array}
\right)  \mbox{ by (i)} \\
&=
\left( 
\begin{array}{c} 
(p-1)+(i-j) \\ 
p-1 
\end{array}
\right) \\
& \equiv 
\left\{
\begin{array}{cl}
1 & i = j \\
0 & i \neq j
\end{array}
\right. \mod p .
\end{align*}
\item 
If $i < s$, then the both sides are $0$. 
We will prove this statement for the case $i \geq s$ by double induction on $i,s$. 
If $i=s$, both sides are $1$. If $i=p$, the left hand side is $1$ and the right hand side is 
\begin{align*}
(-1)^{p+s} 
\left( 
\begin{array}{c} 
p-1 \\ 
s-1 
\end{array}
\right)
\equiv (-1)^{s+1} (-1)^{s-1} \equiv 1 \mod p 
\quad \mbox{ by Lemma \ref{lem-dihedralclaim} (i)} .
\end{align*}
Then if $i=p$ or $s=i$, the statement holds. 
Suppose that the statement is true for $i \geq k_1$ or $s \geq k_2$. 
When $i = k_1 -1$ and $s = k_2 - 1$, the left hand side is 
\begin{align*}
& \left( \begin{array}{c} p-(k_2 - 1) \\ p- (k_1-1) \end{array} \right) \\
&=
\left( \begin{array}{c} p-k_2 \\ p- (k_1-1) \end{array} \right)
+ \left( \begin{array}{c} p-k_2 \\ p- k_1 \end{array} \right) \mbox{ by Pascal's rule} \\
&\equiv
(-1)^{k_1 + k_2 - 1} \left( \begin{array}{c} k_1 - 2 \\ k_2 -1 \end{array} \right)
+ (-1)^{k_1 + k_2} \left( \begin{array}{c} k_1 - 1 \\ k_2 - 1 \end{array} \right) 
 \mbox{ by assumption of induction}\\
&\equiv
(-1)^{k_1 + k_2} 
\left(
- \left( \begin{array}{c} k_1 - 2 \\ k_2 -1 \end{array} \right)
+ \left( \begin{array}{c} k_1 - 1 \\ k_2 - 1 \end{array} \right)
\right) \\
&\equiv
(-1)^{k_1 + k_2} 
\left( \begin{array}{c} k_1 - 2 \\  k_2  - 2 \end{array} \right)  \mbox{ by Pascal's rule} .
\end{align*}
The last expression is the right hand side of the desired statement for $i= k_1 - 1$ and $j = k_2 - 1$. 
This completes the proof. 
\end{enumerate}
\end{proof}

\begin{proof}[Proof of Theorem \ref{thm-Gmpk}]
The action of $G(m,p|k)$ on $C_p$ given by 
\[
a \cdot n = n+1 \mod p, \quad b \cdot n = k n \mod p
\]
induces an embedding of $G(m,p|k)$ into $S_p$. 
We denote by $\bar{\rho}_1 :  G(m,p | k) \to GL(p,{\mathbb Z})$ the restriction of the permutation representation of $S_p$, 
namely  
\[
\bar{\rho}_1(a) = 
\begin{pmatrix}
0 &    & & & 1 \\
1 & 0 & &\text{\huge{0}} & \\
  &  1 & \ddots & & \\
  &     & \ddots & \ddots & \\
& \text{\huge{0}}     &   &  1& 0  
\end{pmatrix}, \quad 
\bar{\rho}_1(b)_{(i,j)} = 
\left\{
\begin{array}{cl}
1 & j \equiv k i - 1 \mod p \\
0 & \mbox{otherwise}
\end{array}
\right. . 
\]
By the chacter table of $G(m,p|k)$ of Lemma \ref{lem-charactergmpk}, $\bar{\rho}_1$ can be expressed as 
\[
\bar{\rho}_1 = {\bf 1} \oplus \rho_1 \oplus \rho_2 \oplus \cdots \oplus \rho_{\bar{p}} .
\]
Furthermore, we define $\bar{\rho}_2 :  G(m,p | k) \to GL(m,{\mathbb Z})$ by 
\[
\bar{\rho}_2(a) = 
\begin{pmatrix}
1 &  & & \text{\huge{0}} \\
  &  1 &  \\
  &     & \ddots & \\
\text{\huge{0}}   &  & &  1  
\end{pmatrix}, \quad 
\bar{\rho}_2(b) = 
\begin{pmatrix}
1 &   0 & 0 \\
1 & -1 &1  & & \text{\huge{0}} \\
1 & -1 & 0 & 1& \\
\vdots &  \vdots& &\ddots  & \ddots  \\
\vdots &  \vdots& \text{\huge{0}} &  & \ddots & 1 \\
1 & -1 & &    & & 0
\end{pmatrix} . 
\]
It is easy to see that $\bar{\rho}_2$ is a representation with character $\chi_{\bar{\rho}_2} = (m,0,\ldots,0,m,\ldots,m)$. 
Then we get 
\[
\bar{\rho}_2 = {\bf 1} \oplus \tau_1 \oplus \tau_2 \oplus \cdots \oplus \tau_{m-1}.  
\]
Since the irreducible decomposition of the regular representation $\rho$ is 
\[
\rho = {\bf 1} \oplus \tau_1 \oplus \tau_2 \oplus \cdots \oplus \tau_{m-1} 
\oplus {\rho_1}^{\oplus m} \oplus {\rho_2}^{\oplus m} \oplus \cdots \oplus  {\rho_{\bar{p}}}^{\oplus m}, 
\]
we have  the equality: 
\[
\rho \oplus {\bf 1}^{\oplus m} =  {\bar{\rho}_1}{}^{\oplus m} \oplus \bar{\rho}_2. 
\]
In order to get $\Delta_K^{\rho \circ f} (t)$, it is sufficient to compute $\Delta_K^{\bar{\rho}_1 \circ f} (t)$ and $\Delta_K^{\bar{\rho}_2 \circ f} (t)$. 

First, we compute $\Delta_K^{\bar{\rho}_1 \circ f} (t)$. 
We take a $p \times p$ matrix $A_p$ as $A_{p,1}$ in Proof of Theorem \ref{thm-dihedral}. 
Then ${A_p}^{-1} \bar{\rho}_1(a) A_p$ is an upper triangular matrix whose diagonal entries are all $1$ 
(c.f. Proof of Theorem \ref{thm-dihedral}). 
By Lemma \ref{lem-gmpkclaim} (ii), 
the $(i,j)$-entry of ${A_p}^{-1} \bar{\rho}_1(b) A_p$ is the following: 
\begin{align*}
{{A_p}^{-1} \bar{\rho}_1(b) A_p}_{(i,j)} 
&= 
\sum_{t=1}^p \left(
\sum_{s=1}^p 
\left( \begin{array}{c} p-s \\ p-i \end{array} \right)
\delta_{ks-1,t}
\right)
\left( \begin{array}{c} t-1 \\ j-1 \end{array} \right) \\
&= 
\sum_{s=1}^p 
\sum_{t=ks-1} 
\left( \begin{array}{c} p-s \\ p-i \end{array} \right)
\left( \begin{array}{c} t-1 \\ j-1 \end{array} \right) \\
&\equiv  
\sum_{s=1}^p 
\left( \begin{array}{c} p-s \\ p-i \end{array} \right)
\left( \begin{array}{c} ks-2 \\ j-1 \end{array} \right) \\
&\equiv  
\sum_{s=1}^p 
(-1)^{i+s}
\left( \begin{array}{c} i-1 \\ s-1 \end{array} \right)
\left( \begin{array}{c} ks-2 \\ j-1 \end{array} \right)  \mbox{ by Lemma \ref{lem-gmpkclaim} (iii)} \\
& \equiv
\left\{
\begin{array}{ll}
0 & i > j \\ 
k^{i-1} & i = j
\end{array}
\right.  \mod p \mbox{ by Euler's finite difference theorem}.
\end{align*}
Then ${A_p}^{-1} \bar{\rho}_1(b) A_p$ is an upper triangular matrix:
\[
{A_p}^{-1} \bar{\rho}_1(b) A_p \equiv 
\begin{pmatrix}
1 & & & &  \\
& k^1 & & \text{\huge{*}} & \\
& & k^2 & & \\
& & & \ddots & \\
& \text{\huge{0}}& & & k^{p-1}
\end{pmatrix}
\mod p . 
\] 
The normal generators of $G(m,p|k)$ are $b$ or $b^{-1}$. 
Recall that $k$ is a primitive $m$-th root of unity modulo $p$ 
and that $k_j$ be $m$-th roots of unity modulo $p$. 
Therefore the twisted Alexander polynomial associated to $\bar{\rho}_1 \circ f$ is 
\begin{align*}
\Delta_K^{\bar{\rho}_1 \circ f} (t) 
& \equiv  \left(\frac{\Delta_K (t)}{t  - 1} \right) \cdot  \prod_{j=1}^{p-1} \left(\frac{\Delta_K (k^j t)}{ k^j t  - 1} \right)  \\
& \equiv \left(\frac{\Delta_K (t)}{t  - 1} \right) \cdot  \prod_{j=1}^{m} \left(\frac{\Delta_K (k^j t)}{ k^j t  - 1} \right)^{\bar{p}}  \\
& \equiv \left(\frac{\Delta_K (t)}{t  - 1} \right) \cdot  \prod_{j=1}^{m} \left(\frac{\Delta_K (k_j t)}{ k_j t  - 1} \right)^{\bar{p}}  
\mod p 
\end{align*}
where $p = m \bar{p}+1$. 

Moreover, $\bar{\rho}_2$ is the direct sum of the one-dimensional representations ${\bf 1}, \tau_1, \tau_2, \ldots , \tau_{m-1}$,  
Lemma \ref{lem-tapproperties} (ii)(iii) show 
\[
\Delta_K^{\bar{\rho}_2 \circ f} (t) = \prod_{j=1}^m \left(\frac{\Delta_K (\omega^j t)}{\omega^j t - 1} \right) . 
\]

Hence we obtain 
\begin{align*}
\Delta_K^{(\rho \oplus {\bf 1}^{\oplus m}) \circ f} (t) 
&= \Delta_K^{(\bar{\rho}_1{}^{\oplus m} \oplus \bar{\rho}_2) \circ f} (t) \\ 
&\equiv \left(\frac{\Delta_K (t)}{t  - 1} \right)^m \cdot 
\prod_{j=1}^m \left(\frac{\Delta_K (k_j t)}{ k_j t  - 1} \right)^{p-1}  \cdot 
\prod_{j=1}^m \left(\frac{\Delta_K (\omega^j t)}{\omega^j t - 1} \right)
\mod p , \\
\Delta_K^{\rho \circ f} (t) &\equiv 
\prod_{j=1}^m \left(\frac{\Delta_K (k_j t)}{ k_j t  - 1} \right)^{p-1}  \cdot 
\prod_{j=1}^m \left(\frac{\Delta_K (\omega^j t)}{\omega^j t - 1} \right)
\mod p .
\end{align*}
\end{proof}

\begin{example}
For example, we consider $G(3,7|2) \simeq C_3 \ltimes C_7$. 
The character table of $G(3,7|2)$ is given by Table \ref{characterG372}. 
We denote by $\rho : G(3,7|2) \to GL(21,{\mathbb Z})$ the regular representation. 
It is easy to see that the $3$-rd roots of unity modulo $7$ are $1,2,4$. 
Suppose that there exists a surjective homomorphism $f : G(K) \to G(3,7|2)$, then 
\begin{align*}
\Delta_K^{\rho \circ f}  (t) & \equiv 
\left(\frac{\Delta_K (t)}{t-1} \right)^{\hspace*{-1mm}6}
\left(\frac{\Delta_K (2 t)}{2 t-1} \right)^{\hspace*{-1mm}6}
\left(\frac{\Delta_K (4 t)}{4 t-1} \right)^{\hspace*{-1mm}6}
\left(\frac{\Delta_K (t)}{t-1} \cdot \frac{\Delta_K (\alpha t)}{\alpha t-1} \cdot \frac{\Delta_K (\alpha^2 t)}{\alpha^2 t-1} \right) \\
& \equiv 
\left(\frac{\Delta_K (t)}{t-1} \right)^{\hspace*{-1mm}7}
\left(\frac{\Delta_K (2 t)}{2 t-1} \right)^{\hspace*{-1mm}6}
\left(\frac{\Delta_K (4 t)}{4 t-1} \right)^{\hspace*{-1mm}6}
\left(\frac{\Delta_K (\alpha t)}{\alpha t-1} \cdot \frac{\Delta_K (\alpha^2 t)}{\alpha^2 t-1} \right)
\mod 7 
\end{align*}
where $\alpha = -\frac{1}{2} + \frac{\sqrt{3}}{2}i \in {\mathbb C}$. 
\begin{table}[h]
\begin{tabular}{c|ccccc}
& $1$ & $b$ & $b^2$ & $a^2$ & $a^3$ \\
\hline
$\chi_{{\bf 1}}$ & $1$ & $1$ & $1$ & $1$ & $1$ \\
$\chi_{\tau_1} $ & $1$ & $- \frac{1}{2} + \frac{\sqrt{3}}{2} i$ & $- \frac{1}{2} - \frac{\sqrt{3}}{2} i$ & $1$ & $1$ \\
$\chi_{\tau_2} $ & $1$ & $- \frac{1}{2} - \frac{\sqrt{3}}{2} i$ & $- \frac{1}{2} + \frac{\sqrt{3}}{2} i$ & $1$ & $1$ \\
$\chi_{\rho_1}$ & $3$ & $0$ & $0$ & $- \frac{1}{2} + \frac{\sqrt{7}}{2} i$ & $- \frac{1}{2} - \frac{\sqrt{7}}{2} i$ \\
$\chi_{\rho_2}$ & $3$ & $0$ & $0$ & $- \frac{1}{2} - \frac{\sqrt{7}}{2} i$ & $- \frac{1}{2} + \frac{\sqrt{7}}{2} i$ 
\end{tabular}
\caption{Character table of $G(3,7|2)$}
\label{characterG372}
\end{table} 
\end{example}

\section{Dicyclic groups $\mbox{Dic}_n$}

For $n \geq 2$, the dicyclic group $\mbox{Dic}_n$ of order $4n$ is defined by the presentation: 
\[
\mbox{Dic}_n = \langle a,b \, | \, a^{2n} = 1,  b^2 = a^n, bab^{-1} = a^{-1} \rangle,  
\]
which can be considered as an extension of $C_2$ by $C_{2n}$, 
namely, we have a short exact sequence: 
\[
1 \to C_{2n} \to \mbox{Dic}_n \to C_2 \to 1 .
\]
Furthermore, $\mbox{Dic}_n$ is also called the binary dihedral group, since $\mbox{Dic}_n / \langle b^2 \rangle$ is isomorphic to $D_n$.  
If $n$ is even, $\mbox{Dic}_n$ is not normally generated by one element.  
Then we deal with the dicyclic group of odd degree in this paper. 
In this case, $\mbox{Dic}_n$ is normally generated by $b$ or $ab$. 

\begin{lemma}\label{lem-dicyclicrep}
For an odd natural number $n$, 
there are one-dimensional irreducible representations ${\bf 1}, \tau_1,\tau_2, \tau_3$ 
and two-dimensional irreducible representations $\rho_{1,j},\rho_{2,j}$ $(1 \leq j \leq \frac{n-1}{2})$ of ${\rm Dic}_n$. 
These representations are defined by the following: 
\[
\begin{array}{ll}
{\bf 1} (a) = 1,  &{\bf 1} (b) = 1, \\
\tau_1 (a) = 1,  &\tau_1 (b) = -1, \\
\tau_2 (a) = -1,  &\tau_2 (b) = i, \\
\tau_3 (a) = -1,  &\tau_3 (b) = -i, \\
\rho_{1,j}(a) = 
\begin{pmatrix}
\bar{\omega}^{2 j} & 0 \\
0 & \bar{\omega}^{-2 j} 
\end{pmatrix},
&\rho_{1,j}(b) = 
\begin{pmatrix}
0 & 1 \\
1 & 0 
\end{pmatrix},\\
\rho_{2,j}(a) = 
\begin{pmatrix}
\bar{\omega}^{2 j-1} & 0 \\
0 & \bar{\omega}^{-(2 j-1)} 
\end{pmatrix},
&\rho_{2,j}(b) = 
\begin{pmatrix}
0 & -1 \\
1 & 0 
\end{pmatrix}.
\end{array}
\]
where $\bar{\omega}=e^{\pi i/n}$.
These characters are given by Table \ref{characterDicn}, where $1 \leq k \leq \frac{n-1}{2}$. 

\begin{table}[h]
\begin{tabular}{c|cccccc}
& $1$ & $a^n$ & $b$ & $ab$ & $a^{2k}$ & $a^{2k-1}$\\
\hline
$\chi_{{\bf 1}}$ & $1$ & $1$ & $1$ & $1$ & $1$ & $1$ \\
$\chi_{\tau_1} $ & $1$ & $1$ & $-1$ & $-1$ & $1$ & $1$ \\
$\chi_{\tau_2} $ & $1$ & $-1$ & $i$ & $-i$ & $1$ & $-1$ \\
$\chi_{\tau_3} $ & $1$ & $-1$ & $-i$ & $i$ & $1$ & $-1$ \\
$\chi_{\rho_{1,j}}$ & $2$ & $2$ & $0$ & $0$ & $\bar{\omega}^{4j k} + \bar{\omega}^{-4j k} $ & $\bar{\omega}^{2j (2k-1)} + \bar{\omega}^{-2j (2k-1)} $ \\
$\chi_{\rho_{2,j}}$ & $2$ & $-2$ & $0$ & $0$ & $\bar{\omega}^{2(2j-1) k} + \bar{\omega}^{-2(2j-1) k} $ & $\bar{\omega}^{(2j-1) (2k-1)} + \bar{\omega}^{-(2j-1)(2k-1)} $ 
\end{tabular}
\caption{Character table of ${\rm Dic}_n$}
\label{characterDicn}
\end{table}
\end{lemma}

By Lemma \ref{lem-dicyclicrep}, we obtain the formula of the twisted Alexander polynomials for the dicyclic group. 

\begin{theorem}\label{thm-dicyclic}
Let $p$ be an odd prime number and $q=p^n$. 
We denote by $\rho : {\rm Dic}_{q} \to GL(4 q,{\mathbb Z})$ the regular representation of ${\rm Dic}_q$. 
If there exists a surjective homomorphism $f :G(K) \to {\rm Dic}_q$, then
\[
\Delta_K^{\rho \circ f} (t) \equiv 
\left(\frac{\Delta_K (t)}{t-1} \cdot 
\frac{\Delta_K (- t)}{t+1}  \cdot 
\frac{\Delta_K (i t)}{i t-1} \cdot \frac{\Delta_K (-i t)}{i t+1} \right)^{q}
\mod p .
\]
\end{theorem}

\begin{proof}[Proof of Theorem \ref{thm-dicyclic}] 
First, Lemma \ref{lem-dicyclicrep} shows that the regular representation $\rho$ of $\mbox{Dic}_n$ is
\[
\rho = {\bf 1} \oplus \tau_1 \oplus \tau_2 \oplus \tau_3 
\oplus {\rho_{1,1}}^{\oplus 2} \oplus \cdots \oplus {\rho_{1,\frac{q-1}{2}}}^{\oplus 2}  
\oplus {\rho_{2,1}}^{\oplus 2} \oplus \cdots \oplus {\rho_{2,\frac{q-1}{2}}}^{\oplus 2} .
\] 
Let $\bar{\rho}$ be a representation defined by 
\[
\bar{\rho} = {\bf 1} \oplus \tau_1 \oplus {\rho_{1,1}}^{\oplus 2} \oplus \cdots \oplus {\rho_{1,\frac{q-1}{2}}}^{\oplus 2},
\]
which is the composition of 
the projection $\mbox{Dic}_q \to \mbox{Dic}_q / \langle b^2 \rangle \simeq D_q$ 
and the regular representation of $D_q$.  
By Theorem \ref{thm-dihedral}, the twisted Alexander polynomial associated to $\bar{\rho} \circ f$ is 
\[
\Delta_K^{\bar{\rho} \circ f} (t) \equiv 
\left(
\frac{\Delta_K (t)}{t-1} \cdot \frac{\Delta_K (-t)}{t+1}
\right)^{q} \mod p .
\]
By Table \ref{characterDicn}, we get    
\[
\tau_2 \otimes {\bf 1} = \tau_2 , \quad \tau_2 \otimes \tau_1 = \tau_3 , \quad \tau_2 \otimes \rho_{1,j} =\rho_{2,\frac{q-1}{2} - j}  
\]
up to conjugation. 
Then by Lemma \ref{lem-tapproperties} (iii), the twisted Alexander polynomial associated to $(\tau_2 \otimes \bar{\rho}) \circ f$ is
\[
\Delta_K^{(\tau_2 \otimes \bar{\rho}) \circ f} (t) \equiv 
\left(
\frac{\Delta_K (it)}{it-1} \cdot \frac{\Delta_K (-it)}{it+1}
\right)^{q} \mod p .
\]
Hence we obtain the statement 
 \begin{align*}
&\Delta_K^{\rho \circ f} (t) \equiv 
\Delta_K^{(\bar{\rho} \oplus (\tau_2 \otimes \bar{\rho})) \circ f} (t) \equiv \\
& 
\left(\frac{\Delta_K (t)}{t-1} \cdot 
\frac{\Delta_K (- t)}{t+1} \cdot 
\frac{\Delta_K (i t)}{i t-1} \cdot \frac{\Delta_K (-i t)}{i t+1} \right)^{q}
\mod p .
\end{align*}
\end{proof}

\section{Concluding remarks}

In \cite{FV13-1}, Friedl-Vidussi showed that for a non-fibered knot $K$, 
there exists a finite group $G$ and a surjective homomorphism $f : G(K) \to G$ such that $\Delta_K^{\rho \circ f} (t)=0$. 
Then in \cite{ishikawamorifujisuzuki} we define the twisted Alexander vanishing (TAV) order ${\mathcal O}(K)$ of a knot $K$ 
as the smallest order of a finite group $G$ such that 
there exists a surjective homomorphism $f : G(K) \to G$ with $\Delta_K^{\rho \circ f} (t)=0$, 
which is called the minimal order in \cite{morifujisuzuki}. 

By Friedl-Vidussi's Theorem \cite{FV13-1}, ${\mathcal O}(K)$ is finite for any non-fibered knot $K$. 
On the other hand, we define ${\mathcal O}(K) = + \infty$ for a fibered knot $K$. 
In general, it does not seem to be easy to determine ${\mathcal O}(K)$ for a given non-fibered knot $K$. 
However, we obtain the explicit values for some knots in \cite{morifujisuzuki}, \cite{ishikawamorifujisuzuki}. 
Moreover, the lower and upper bound of ${\mathcal O} (K)$ are studied in \cite{ishikawamorifujisuzuki}. 
To be precise,  ${\mathcal O} (K)$ is not upper bounded for non-fibered knots and we have the following for the lower bound. 

\begin{theorem}\cite{ishikawamorifujisuzuki}\label{thm-OK}
For any knot $K$, we have ${\mathcal O} (K) \geq 24$. 
\end{theorem}

In this section, we will give alternative proof of Theorem \ref{thm-OK}. 
It is known that 
$35$ finite groups are normally generated by one element, out of $59$ finite groups of order less than $24$. 
We can apply 
Proposition \ref{prop-abelian} for $23$ abelian finite groups, 
Theorem \ref{thm-dihedral} for $D_3,D_5,D_7,D_9,D_{11}$, 
Corollary \ref{cor-dihedralcyclic} for $D_3 \times C_3$, 
Theorem \ref{thm-Gmpk} for $C_4 \ltimes C_5 = G(4,5|2), C_3 \ltimes C_7 = G(3,7|2)$, 
and Theorem \ref{thm-dicyclic} for ${\rm Dic}_3, {\rm Dic}_5$. 
They follow that the twisted Alexander polynomials are not zero for any knots. 
Then the remaining groups are only $A_4, D_3 \ltimes C_3$ out of $59$ finite groups. 

In this section, we also provide explicit formulas of the twisted Alexander polynomials for $A_4, D_3 \ltimes C_3$, 
which imply that the twisted Alexander polynomials are not zero for any knots. 

\begin{proposition}\label{prop-G12_3}
We denote by $\rho : A_4 \to GL(12,{\mathbb Z})$ the regular representation of $A_4$. 
If there exists a surjective homomorphism $f :G(K) \to A_4$, then
\[
\Delta_K^{\rho \circ f} (t) \equiv 
\left(\frac{\Delta_K (t)}{t-1} \cdot \frac{\Delta_K (\alpha t)}{\alpha t-1} \cdot \frac{\Delta_K (\alpha^2 t)}{\alpha^2 t-1} \right)^4
\mod 2 
\]
where $\alpha = -\frac{1}{2} + \frac{\sqrt{3}}{2}i \in {\mathbb C}$. 
\end{proposition}

\begin{proof}
The alternating group $A_4$ admits a presentation 
\[
A_4 = \langle a,b\, | \, a^3 = b^2 = 1, (ab)^3 = 1 \rangle.
\]
It is known that the irreducible representations of $A_4$ are given by 
\[
\begin{array}{ll}
{\bf 1} (a) = 1, & {\bf 1} (b) = 1, \\
\tau_1 (a) = \alpha, & \tau_1 (b) = 1, \\
\tau_2 (a) = \alpha^2, & \tau_2 (b) = 1, \\
\tau_3 (a) = 
\begin{pmatrix}
-1 & 1 & 0 \\
-1 & 0 & 1 \\
0 & 0 & 1 
\end{pmatrix}, & 
\tau_3 (b) = 
\begin{pmatrix}
-1 & 1 & 0 \\
0 & 1 & 0 \\
0 & 1 & -1
\end{pmatrix}. 
\end{array} 
\]
The regular representation $\rho$ can be expressed as  
$\rho = {\bf 1} \oplus \tau_1 \oplus \tau_2 \oplus \tau_3^{\oplus 3}$. 
Since $f :G(K) \to A_4$ is a surjective homomorphism and the normal generators of $A_4$ are $a$ and $a^2$, 
we may assume that $f(x_i) = a$ or $a^2$ or their conjugate.  
Let $P$ be a $3 \times 3$ matrix given by 
\[
\begin{pmatrix}
1 & 0 & 0 \\
0 & 1 & 0 \\
1 & 0 & 1 
\end{pmatrix}
\]
and define $\tau'_3(g) = P^{-1} \tau_3(g) P \mod 2$ for $g \in A_4$. 
In particular, we have 
\[
\tau'_3(a) \equiv 
\begin{pmatrix}
1 & 1 & 0 \\
0 & 0 & 1 \\
0 & 1 & 1 
\end{pmatrix}, \quad 
\tau'_3(a^2) \equiv 
\begin{pmatrix}
1 & 1 & 1 \\
0 & 1 & 1 \\
0 & 1 & 0 
\end{pmatrix}, \quad 
\tau'_3(b) \equiv 
\begin{pmatrix}
1 & 1 & 0 \\
0 & 1 & 0 \\
0 & 0 & 1 
\end{pmatrix} \mod 2 .
\]
Moreover, for any elements $g_1$ and $g_2$ in the conjugacy classes of $a$ and $a^2$ respectively, we get 
\[
\tau'_3(g_1) \equiv 
\left( \begin{array}{c|cc}
1 & * & * \\ \hline
0 & 0 & 1 \\
0 & 1 & 1 
\end{array} \right), \quad 
\tau'_3(g_2) \equiv 
\left( \begin{array}{c|cc}
1 & * & * \\ \hline
0 & 1 & 1 \\
0 & 1 & 0 
\end{array} \right) \mod 2.
\]
We put 
\[
A = 
\begin{pmatrix}
0 & 1 \\
-1 & -1
\end{pmatrix}, \quad 
B = 
\begin{pmatrix}
-1 & -1 \\
1 & 0
\end{pmatrix} \in SL(2,{\mathbb Z}) ,
\]
then $\tau'_3(g_1) ,\tau'_3(g_2) $ can be considered as 
\[
\tau'_3(g_1) \equiv 
\left( \begin{array}{c|c}
1 & *  \\ \hline
0 & A 
\end{array} \right), \quad 
\tau'_3(g_2) \equiv 
\left( \begin{array}{c|c}
1 & *  \\ \hline
0 & B 
\end{array} \right) \mod 2.
\]
Let $\sigma_1, \sigma_2 : G(K) \to SL(2,{\mathbb Z})$ be representations of $G(K)$ 
defined by $\sigma_1(x_i) = A, \sigma_2 (x_i) = B$ respectively.   
Since the eigenvalues of $A$ and $B$ are both $\alpha,\alpha^2$,  
the twisted Alexander polynomials associated to $\sigma_1 \circ f, \sigma_2 \circ f$ are 
\[
\Delta_K^{\sigma_1 \circ f} (t) = \Delta_K^{\sigma_2 \circ f} (t) = 
\left( \frac{\Delta_K (\alpha t)}{\alpha t-1} \cdot \frac{\Delta_K (\alpha^2 t)}{\alpha^2 t-1} \right).
\]
Remark that $\Delta_K (\alpha t) \Delta_K (\alpha^2 t)$ and $(\alpha t-1) (\alpha^2 t-1)$ are polynomials over integers. 
By Lemma \ref{lem-tapproperties} (ii), the twisted Alexander polynomial associated to $\tau_3 \circ f$ is
 \[
\Delta_K^{\tau_3 \circ f} (t) \equiv 
\Delta_K^{\tau'_3 \circ f} (t) \equiv 
\left(\frac{\Delta_K (t)}{t-1} \cdot \frac{\Delta_K (\alpha t)}{\alpha t-1} \cdot \frac{\Delta_K (\alpha^2 t)}{\alpha^2 t-1} \right)
\mod 2 .
\]
Hence 
\[
\Delta_K^{\rho \circ f} (t) \equiv 
\Delta_K^{({\bf 1} \oplus \tau_1 \oplus \tau_2 \oplus \tau_3^{\oplus 3}) \circ f} (t) \equiv 
\left(\frac{\Delta_K (t)}{t-1} \cdot \frac{\Delta_K (\alpha t)}{\alpha t-1} \cdot \frac{\Delta_K (\alpha^2 t)}{\alpha^2 t-1} \right)^4
\mod 2 .
\]
\end{proof}

Note that Hirasawa-Murasugi in \cite{hirasawamurasugi2} discussed the twisted Alexander polynomial for $A_4$
associated to the permutation representation of $A_4$. 

\begin{proposition}\label{prop-G18_4}
We denote by $\rho : D_3 \ltimes C_3 \to GL(18,{\mathbb Z})$ the regular representation of $D_3 \ltimes C_3$. 
If there exists a surjective homomorphism $f :G(K) \to D_3 \ltimes C_3$, then
\[
\Delta_K^{\rho \circ f} (t) \equiv 
\left(\frac{\Delta_K (t)}{t-1} \cdot 
\frac{\Delta_K (- t)}{t+1} \right)^{9} 
\mod 3 .
\]
\end{proposition}

\begin{proof}
The finite group $D_3 \ltimes C_3$ admits a presentation 
\[
D_3 \ltimes C_3 = 
\langle a, b, c \, | \, a^3 = b^3 = c^2 = 1, ab = ba , cac = a^{-1}, cbc = b^{-1} \rangle. 
\]
It is known that the irreducible representations of $D_3 \ltimes C_3$ are ${\bf 1}, \tau, \rho_1, \rho_2, \rho_3, \rho_4$ and that 
their characters are give by Table \ref{characterD3C3} (c.f. \cite{groupnames}). 

\begin{table}[h]
\begin{tabular}{c|cccccc}
& $1$ & $c$ & $a$ & $b$ & $ab$ & $a^2b$\\
\hline
$\chi_{{\bf 1}}$ & $1$ & $1$ & $1$ & $1$ & $1$ & $1$ \\
$\chi_{\tau} $ & $1$ & $-1$ & $1$ & $1$ & $1$ & $1$ \\
$\chi_{\rho_1}$ & $2$ & $0$ & $-1$ & $-1$ & $-1$ & $2$ \\
$\chi_{\rho_2}$ & $2$ & $0$ & $2$ & $-1$ & $-1$ & $-1$ \\
$\chi_{\rho_3}$ & $2$ & $0$ & $-1$ & $2$ & $-1$ & $-1$ \\
$\chi_{\rho_4}$ & $2$ & $0$ & $-1$ & $-1$ & $2$ & $-1$ \\
\end{tabular}
\caption{Character table of $D_3 \ltimes C_3$}
\label{characterD3C3}
\end{table}

By considering the generators $a,b,c$ as elements $(123)(456)(789), (158)(269)(347)$, $(23)(49)(58)(67)$ in $S_9$ respectively, 
$D_3 \ltimes C_3$ can be embedded in $S_9$. 
We denote by $\bar{\rho} : D_3 \ltimes C_3  \to GL(9, {\mathbb Z})$ the restriction of the permutation representation of $S_9$, namely, 
\begin{small}
\begin{align*}
\bar{\rho} (a) &= 
\left(
 \begin{array}{ccccccccc}
  0 & 1 & 0 & 0 & 0 & 0 & 0 & 0 & 0 \\
 0 & 0 & 1 & 0 & 0 & 0 & 0 & 0 & 0 \\
 1 & 0 & 0 & 0 & 0 & 0 & 0 & 0 & 0 \\
 0 & 0 & 0 & 0 & 1 & 0 & 0 & 0 & 0 \\
 0 & 0 & 0 & 0 & 0 & 1 & 0 & 0 & 0 \\
 0 & 0 & 0 & 1 & 0 & 0 & 0 & 0 & 0 \\
 0 & 0 & 0 & 0 & 0 & 0 & 0 & 1 & 0 \\
 0 & 0 & 0 & 0 & 0 & 0 & 0 & 0 & 1 \\
 0 & 0 & 0 & 0 & 0 & 0 & 1 & 0 & 0 \\
\end{array}
 \right), 
 \bar{\rho} (b) = 
 \left(
 \begin{array}{ccccccccc}
  0 & 0 & 0 & 0 & 1 & 0 & 0 & 0 & 0 \\
 0 & 0 & 0 & 0 & 0 & 1 & 0 & 0 & 0 \\
 0 & 0 & 0 & 1 & 0 & 0 & 0 & 0 & 0 \\
 0 & 0 & 0 & 0 & 0 & 0 & 1 & 0 & 0 \\
 0 & 0 & 0 & 0 & 0 & 0 & 0 & 1 & 0 \\
 0 & 0 & 0 & 0 & 0 & 0 & 0 & 0 & 1 \\
 0 & 0 & 1 & 0 & 0 & 0 & 0 & 0 & 0 \\
 1 & 0 & 0 & 0 & 0 & 0 & 0 & 0 & 0 \\
 0 & 1 & 0 & 0 & 0 & 0 & 0 & 0 & 0 \\
\end{array}
 \right), \\
 \bar{\rho} (c) &= 
 \left(
 \begin{array}{ccccccccc}
  1 & 0 & 0 & 0 & 0 & 0 & 0 & 0 & 0 \\
 0 & 0 & 1 & 0 & 0 & 0 & 0 & 0 & 0 \\
 0 & 1 & 0 & 0 & 0 & 0 & 0 & 0 & 0 \\
 0 & 0 & 0 & 0 & 0 & 0 & 0 & 0 & 1 \\
 0 & 0 & 0 & 0 & 0 & 0 & 0 & 1 & 0 \\
 0 & 0 & 0 & 0 & 0 & 0 & 1 & 0 & 0 \\
 0 & 0 & 0 & 0 & 0 & 1 & 0 & 0 & 0 \\
 0 & 0 & 0 & 0 & 1 & 0 & 0 & 0 & 0 \\
 0 & 0 & 0 & 1 & 0 & 0 & 0 & 0 & 0 \\
\end{array}
 \right) .
\end{align*}
\end{small}
The character table gives the irreducible decomposition $\bar{\rho} = {\bf 1} \oplus \rho_1 \oplus \rho_2 \oplus \rho_3 \oplus \rho_4$. 
Since the regular representation admits the irreducible decomposition: 
$\rho = {\bf 1} \oplus \tau \oplus \rho_1^{\oplus 2} \oplus \rho_2^{\oplus 2} \oplus \rho_3^{\oplus 2} \oplus \rho_4^{\oplus 2}$,   
we have the following equality: 
\[
\rho \oplus {\bf 1} = \bar{\rho}^{\oplus 2} \oplus \tau .
\]

We define a $9 \times 9$ matrix $P$ as follows:
\[
P \equiv 
 \left(
 \begin{array}{ccccccccc}
 1 & 0 & 0 & 0 & 0 & 0 & 0 & 0 & 0 \\
 1 & 1 & 0 & 0 & 0 & 0 & 0 & 0 & 0 \\
 1 & 2 & 1 & 0 & 0 & 0 & 0 & 0 & 0 \\
 1 & 0 & 0 & 1 & 0 & 0 & 0 & 0 & 0 \\
 1 & 1 & 0 & 1 & 1 & 0 & 0 & 0 & 0 \\
 1 & 2 & 1 & 1 & 2 & 1 & 0 & 0 & 0 \\
 1 & 1 & 0 & 2 & 2 & 0 & 1 & 0 & 0 \\
 1 & 2 & 1 & 2 & 1 & 2 & 1 & 1 & 0 \\
 1 & 0 & 0 & 2 & 0 & 0 & 1 & 2 & 1 \\
\end{array}
 \right) \mod 3 .
\]
Then we have $\bar{\rho}' (g) = P^{-1} \bar{\rho}(g) P$: 
 \begin{small}
\begin{align*}
\bar{\rho}' (a) &\equiv 
\left(
 \begin{array}{ccccccccc}
  1 & 1 & 0 & 0 & 0 & 0 & 0 & 0 & 0 \\
 0 & 1 & 1 & 0 & 0 & 0 & 0 & 0 & 0 \\
 0 & 0 & 1 & 0 & 0 & 0 & 0 & 0 & 0 \\
 0 & 0 & 0 & 1 & 1 & 0 & 0 & 0 & 0 \\
 0 & 0 & 0 & 0 & 1 & 1 & 0 & 0 & 0 \\
 0 & 0 & 0 & 0 & 0 & 1 & 0 & 0 & 0 \\
 0 & 0 & 0 & 0 & 0 & 0 & 1 & 1 & 0 \\
 0 & 0 & 0 & 0 & 0 & 0 & 0 & 1 & 1 \\
 0 & 0 & 0 & 0 & 0 & 0 & 0 & 0 & 1 \\
\end{array}
 \right), 
 \bar{\rho}' (b) \equiv 
 \left(
 \begin{array}{ccccccccc}
  1 & 1 & 0 & 1 & 1 & 0 & 0 & 0 & 0 \\
 0 & 1 & 1 & 0 & 1 & 1 & 0 & 0 & 0 \\
 0 & 0 & 1 & 0 & 0 & 1 & 0 & 0 & 0 \\
 0 & 0 & 0 & 1 & 1 & 0 & 1 & 0 & 0 \\
 0 & 0 & 0 & 0 & 1 & 1 & 0 & 1 & 0 \\
 0 & 0 & 0 & 0 & 0 & 1 & 0 & 0 & 1 \\
 0 & 0 & 0 & 0 & 0 & 0 & 1 & 1 & 0 \\
 0 & 0 & 0 & 0 & 0 & 0 & 0 & 1 & 1 \\
 0 & 0 & 0 & 0 & 0 & 0 & 0 & 0 & 1 \\
\end{array}
 \right), \\
 \bar{\rho}' (c) &\equiv 
 \left(
 \begin{array}{ccccccccc}
  1 & 0 & 0 & 0 & 0 & 0 & 0 & 0 & 0 \\
 0 & 2 & 1 & 0 & 0 & 0 & 0 & 0 & 0 \\
 0 & 0 & 1 & 0 & 0 & 0 & 0 & 0 & 0 \\
 0 & 0 & 0 & 2 & 0 & 0 & 1 & 2 & 1 \\
 0 & 0 & 0 & 0 & 1 & 2 & 0 & 2 & 2 \\
 0 & 0 & 0 & 0 & 0 & 2 & 0 & 0 & 1 \\
 0 & 0 & 0 & 0 & 0 & 0 & 1 & 1 & 0 \\
 0 & 0 & 0 & 0 & 0 & 0 & 0 & 2 & 0 \\
 0 & 0 & 0 & 0 & 0 & 0 & 0 & 0 & 1 \\
\end{array}
 \right) \mod 3.
\end{align*}
\end{small}
Since $f :G(K) \to D_3 \ltimes C_3$ is surjective, $f(x_i)$ is $c$ or the conjugate. 
Therefore the twisted Alexander polynomial associated to $\bar{\rho} \circ f$ is 
\[
\Delta_K^{\bar{\rho} \circ f} (t) \equiv 
\Delta_K^{\bar{\rho}' \circ f} (t) \equiv 
\left(\frac{\Delta_K (t)}{t-1} \right)^{5}  \cdot 
\left(\frac{\Delta_K (- t)}{t+1} \right)^{4} 
\mod 3 .
\]
Note that the twisted Alexander polynomials associated to ${\bf 1} \circ f$ and $\tau \circ f$ are 
\[
\Delta_K^{{\bf 1} \circ f} (t) = \frac{\Delta_K (t)}{t-1}, \quad \Delta_K^{\tau \circ f} (t) = \frac{\Delta_K (-t)}{t+1}.
\]

Hence we obain 
\[
\Delta_K^{(\rho \oplus {\bf 1})  \circ f} (t) = 
\Delta_K^{(\bar{\rho}^{\oplus 2} \oplus \tau)  \circ f} (t) \equiv 
\left(\frac{\Delta_K (t)}{t-1} \right)^{10}  \cdot 
\left(\frac{\Delta_K (- t)}{t+1} \right)^{9} 
\mod 3 
\]
and the twisted Alexander polynomial associated to $\rho \circ f$ is 
\[
\Delta_K^{\rho \circ f} (t) \equiv 
\left(\frac{\Delta_K (t)}{t-1} \cdot 
\frac{\Delta_K (- t)}{t+1} \right)^{9} 
\mod 3 .
\]
\end{proof}

In particular, 
Proposition \ref{prop-G12_3} and \ref{prop-G18_4} show that the twisted Alexander polynomials for $A_4, D_3 \ltimes C_3$ are not zero for any knots. 
Therefore we obtain alternative proof for Theorem \ref{thm-OK}. 
Finally, we propose a conjecture which is a generalization of Proposition \ref{prop-G18_4}. 

\begin{conjecture}
We denote by $\rho : D_p \ltimes C_p \to GL(2 p^2,{\mathbb Z})$ the regular representation of $D_p \ltimes C_p$. 
If there exists a surjective homomorphism $f :G(K) \to D_p \ltimes C_p$, then
\[
\Delta_K^{\rho \circ f} (t) \equiv 
\left(\frac{\Delta_K (t)}{t-1} \cdot 
\frac{\Delta_K (- t)}{t+1} \right)^{p^2} 
\mod p .
\]
\end{conjecture}

\subsection*{Acknowledgments}
The authors are supported in part by 
JSPS KAKENHI Grant Numbers JP20K03596 and JP21K03253. 
A part of this research was done during their stay 
at the Research Institute for Mathematical Sciences, Kyoto University. 
They would like to express their sincere thanks for hospitality. 
Furthermore, they also would like to thank Katsumi Ishikawa for helpful comments. 

\bibliographystyle{amsplain}

\begin{thebibliography}{30}

\bibitem{BF14-1}
H. Boden and S. Friedl, 
\textit{Metabelian $SL (n,{\mathbb C})$ representations of knot groups IV: 
twisted Alexander polynomials}, 
Math. Proc. Cambridge Philos. Soc. {\bf 156} (2014), 81--97.

\bibitem{groupnames}
T. Dokchitser, 
\textit{https://people.maths.bris.ac.uk/~matyd/GroupNames/index.html} 

\bibitem{fox} 
R. Fox, 
\textit{Metacylic invariants of knots and links}, 
Canad. J. Math. {\bf 22} (1970), 193--201.

\bibitem{FV07}
S. Friedl and S. Vidussi,
\textit{Nontrivial Alexander polynomials of knots and links}, 
Bull. London Math. Soc. {\bf 39} (2007), 614–622.

\bibitem{FV13-1}
S. Friedl and S. Vidussi,
\textit{A vanishing theorem for twisted Alexander polynomials with 
applications to symplectic $4$-manifolds}, 
J. Eur. Math. Soc. {\bf 15} (2013), 2027--2041.

\bibitem{FV17}
S. Friedl and S. Vidussi,
\textit{Twisted Alexander invariants detect trivial links}, 
Canad. Math. Bull. {\bf 60} (2017), 283-299.

\bibitem{hirasawamurasugi}
M. Hirasawa and K. Murasugi, 
\textit{Twisted Alexander polynomials of 2-bridge knots associated to metacyclic representations}, 
preprint, arXiv math.GT 0903.0147.

\bibitem{hirasawamurasugi2}
M. Hirasawa and K. Murasugi, 
\textit{Twisted Alexander polynomials of 2-bridge knots associated to metabelian representations}, 
preprint, arXiv math.GT 0903.1689.

\bibitem{ishikawamorifujisuzuki}
K. Ishikawa, T. Morifuji, and M. Suzuki,
\textit{Twisted Alexander vanishing order of knots},
preprint, arXiv math.GT 2310.10936.

\bibitem{gould}
Q. Jocelyn and H. Gould, 
\textit{Combinatorial identities for Stirling numbers}
World Scientific Publishing Co. Pte. Ltd., Singapore, (2016).

\bibitem{KSW05-1}
T. Kitano, M. Suzuki, and M. Wada,
\textit{Twisted Alexander polynomials and surjectivity of a group homomorphism},
Algebr. Geom. Topol. {\bf 5} (2005), 1315--1324.
Erratum: Algebr. Geom. Topol. {\bf 11} (2011), 2937--2939.

\bibitem{morifujisuzuki}
T. Morifuji and M. Suzuki,
\textit{On a theorem of Friedl and Vidussi}, 
Internat. J. Math. {\bf 33} (2022), No.2250085, 14 pages.

\bibitem{Rolfsen76-1}
D. Rolfsen, 
\textit{Knots and links}, Publish or Perish Inc. Berkeley, Calif. (1976). 

\bibitem{SW06-1}
D. Silver and W. Whitten, 
\textit{Knot group epimorphisms}, 
J. Knot Theory Ramifications {\bf 15} (2006), 153--166. 

\bibitem{SW06-2}
D. Silver and S. Williams, 
\textit{Twisted Alexander polynomials detect the unknot}, 
Algebr. Geom. Topol. {\bf 6} (2006), 1893--1901.

\bibitem{Wada94-1}
M. Wada,
\textit{Twisted Alexander polynomial for finitely
presentable groups},
Topology {\bf 33} (1994), 241--256.


%
%
%
%
%
%
%
%
%
%
%
%
%
%

\end{thebibliography}

\end{document}